\documentclass[11pt]{amsart}
\usepackage[margin=1in]{geometry}
\usepackage{latexsym}
\usepackage{amsfonts}
\usepackage{amsmath}
\usepackage{amssymb}
\usepackage{amsthm}
\usepackage{enumerate}
\usepackage[hang,flushmargin]{footmisc}
\usepackage{caption}
\usepackage{tabu}
\usepackage{mathrsfs}
\usepackage{amsaddr}
\usepackage{subfig}

\usepackage{graphicx}
\usepackage{epstopdf}
\usepackage{epsfig}
\usepackage{caption}
 
\usepackage{bm} 
\usepackage{cite}

\newtheorem{theorem}{Theorem}[section]
\newtheorem{lemma}{Lemma}[section]

\newtheorem{conjecture}{Conjecture}[section]
\newtheorem{question}{Question}[section]

\theoremstyle{definition}
\newtheorem{definition}{Definition}[section]

\begin{document}
\title{Connected Components of Complex Divisor Functions}
\author{Colin Defant}
\address{Princeton University \\ Fine Hall, 304 Washington Rd. \\ Princeton, NJ 08544}
\email{cdefant@princeton.edu}

\begin{abstract}
For any complex number $c$, define the divisor function $\sigma_c\colon\mathbb N\to\mathbb C$ by $\displaystyle\sigma_c(n)=\sum_{d\mid n}d^c$. Let $\overline{\sigma_c(\mathbb N)}$ denote the topological closure of the range of $\sigma_c$. Extending previous work of the current author and Sanna, we prove that $\overline{\sigma_c(\mathbb N)}$ has nonempty interior and has finitely many connected components if $\Re(c)\leq 0$ and $c\neq 0$. We end with some open problems.  
\end{abstract}

\maketitle

\bigskip



\section{Introduction}
One of the most famous and perpetually mysterious mathematical contributions of the ancient Greeks is the notion of a \emph{perfect number}, a number that is equal to the sum of its proper divisors. Many early theorems in number theory spawned from attempts to understand perfect numbers. Although few modern mathematicians continue to attribute the same theological or mystical significance to perfect numbers that ancient people once did, these numbers remain a substantial inspiration for research in elementary number theory \cite{Bezuszka, Cohen, Defant2, Hoque, McCranie, Pollack1, Pollack2, Pomerance}.   

A positive integer $n$ is perfect if and only if $\sigma_1(n)/n=2$, where $\sigma_1(n)=\sum_{d\mid n}d$. More generally, for any complex number $c$, we define the divisor function $\sigma_c$ by \[\sigma_c(n)=\sum_{d\mid n}d^c.\] These functions have been studied frequently in the special cases in which $c\in\{-1,0,1\}$. The divisor functions $\sigma_c$ with $c\in\mathbb N$ have also received a fair amount of attention, especially because of their close connections with Eisenstein series \cite{Koblitz, Stein}. Ramanujan, in particular, gave several pleasing identities involving divisor functions \cite{Hardy, Ramanujan}.  

Around 100 A.D., Nicomachus stated that perfect numbers represent a type of ``harmony" between ``deficient" and ``abundant" numbers. A positive integer $n$ is \emph{deficient} if $\sigma_1(n)/n<2$ and is \emph{abundant} if $\sigma_1(n)/n>2$. 
A simple argument shows that $\sigma_1(n)/n=\sigma_{-1}(n)$, so the function $\sigma_{-1}$ is often called the ``abundancy index." The abundancy index is also used to define \emph{friendly} and \emph{solitary} numbers. Two distinct positive integers $m$ and $n$ are said to be friends with each other if $\sigma_{-1}(m)=\sigma_{-1}(n)$. An integer is friendly if it is friends with some other integer, and it is solitary otherwise. For example, any two perfect numbers are friends with each other. There are several extremely difficult problems surrounding the notions of friendly and solitary numbers. For example, it is not known whether $10$ is friendly or solitary \cite{Anderson, Pollack1, Pollack3, Ward}.

\begin{figure}[t]
  \centering
  \subfloat[]{\includegraphics[width=0.42\textwidth]{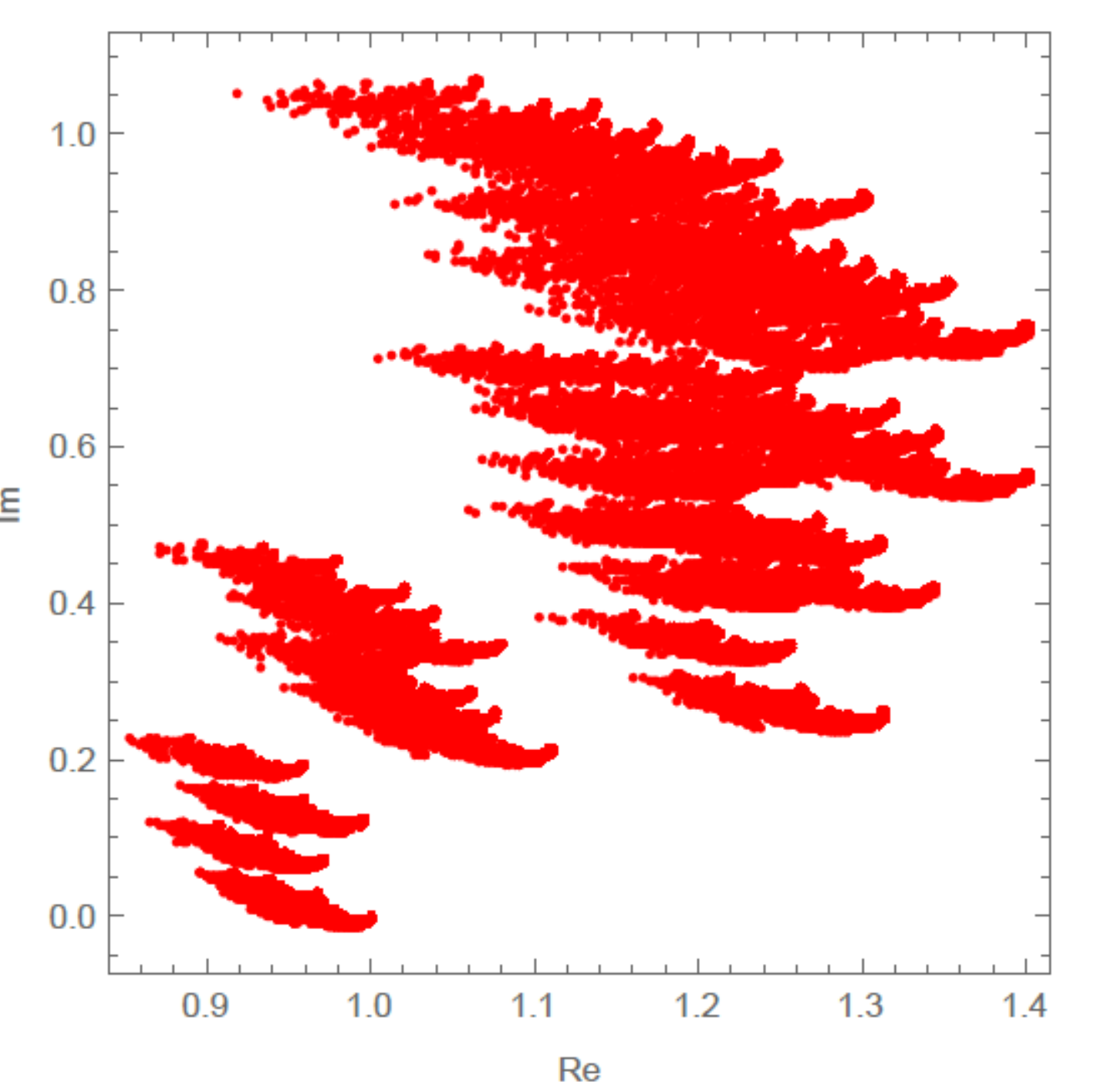}}
  \hfill
  \subfloat[]{\includegraphics[width=0.42\textwidth]{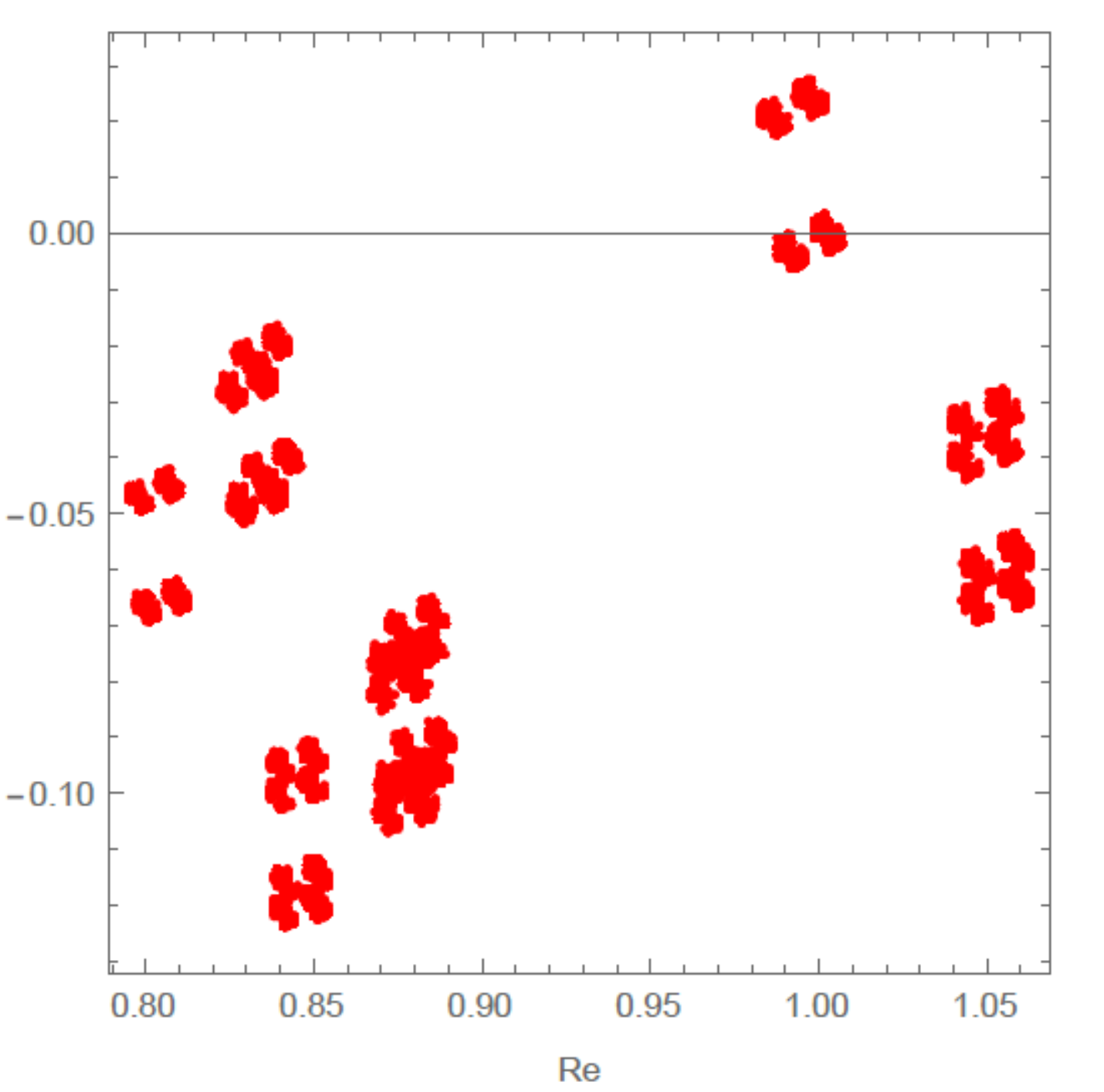}}
  \caption{Plots of $\sigma_c(n)$ for $1\leq n\leq 10^6$. The left image shows a plot when $c=-1.3+i$. The right image shows a plot when $c=-2.3+5i$.}\label{Fig2}
\end{figure}

Motivated by the difficult problems related to perfect and friendly numbers, Laatsch studied $\sigma_{-1}(\mathbb N)$, the range of $\sigma_{-1}$. He showed that $\sigma_{-1}(\mathbb N)$ is a dense subset of the interval $[1,\infty)$ and asked if $\sigma_{-1}(\mathbb N)$ is in fact equal to the set $\mathbb Q\cap[1,\infty)$ \cite{Laatsch}. Weiner answered this question in the negative, showing that $(\mathbb Q\cap[1,\infty))\setminus\sigma_{-1}(\mathbb N)$ is also dense in $[1,\infty)$ \cite{Weiner}. Recently, the author asked what could be said about the topological properties of the ranges of the divisor functions $\sigma_c$ for general complex numbers $c$. More specifically, he has studied sets of the form $\overline{\sigma_c(\mathbb N)}$, where the overline denotes the topological closure. The fractal-like appearances of these sets make them difficult to describe mathematically (see Figure \ref{Fig2}). Nevertheless, the simple number-theoretic definition of the function $\sigma_c$ lends itself to proofs of some nontrivial and often surprising facts concerning $\overline{\sigma_c(\mathbb N)}$. For example, the following theorems summarize some of the main results in this area. Theorem \ref{Thm1} is proven in \cite{Defant4} while the various parts of Theorem \ref{Thm2} are all proven in \cite{Defant1}. In what follows, let $\mathcal N(c)$ denote the (possibly infinite) number of connected components of $\overline{\sigma_c(\mathbb N)}$.  

\begin{theorem}[\cite{Defant4}]\label{Thm1}
Let $r\in\mathbb R$. There exists a constant $\eta\approx 1.8877909$ such that $\mathcal N(-r)=1$ (that is, $\overline{\sigma_{-r}(\mathbb N)}$ is connected) if and only if $r\in(0,\eta]$.  
\end{theorem}

\begin{theorem}[\cite{Defant1}]\label{Thm2}
Let $c\in\mathbb C$. 
\begin{enumerate}
\item The set $\sigma_c(\mathbb N)$ is bounded if and only if $\Re(c)<-1$. 
\item If $\Re(c)>0$ or $c=0$, then $\sigma_c(\mathbb N)$ is discrete. If $\Re(c)\leq 0$ and $c\neq 0$, then $\sigma_c(\mathbb N)$ has no isolated points. 
\item The set $\sigma_c(\mathbb N)$ is dense in $\mathbb C$ if and only if $-1\leq \Re(c)\leq 0$ and $c\not\in\mathbb R$. 
\item We have $\mathcal N(c)\to\infty$ as $\Re(c)\to-\infty$. 
\item If $\Re(c)\leq -3.02$, then $\mathcal N(c)\geq 2$ (so $\overline{\sigma_c(\mathbb N)}$ is disconnected). 
\end{enumerate}
\end{theorem}

The current author has proven analogues of Theorem \ref{Thm1} for other families of arithmetic functions that are relatives of the divisor functions (the so-called ``unitary divisor functions" and ``alternating divisor functions") \cite{Defant3,Defant6}. He has also proven analogues of this theorem in which one considers the images of particular subsets of $\mathbb N$ under $\sigma_{-r}$ \cite{Defant5}. 

Suppose $r>1$ is real. In \cite{Sanna}, Sanna provides an algorithm for computing $\overline{\sigma_{-r}(\mathbb N)}$ when $r$ is known with sufficient precision. As a consequence, he proves that $\mathcal N(-r)<\infty$. This serves as a nice complement to part (4) of Theorem \ref{Thm2}, which implies that $\mathcal N(-r)\to\infty$ as $r\to\infty$. Very recently, Zubrilina \cite{Nina} has obtained asymptotic estimates for $\mathcal N(-r)$, showing that \[\pi(r)+1\leq\mathcal N(-r)\leq\frac{1}{2}\exp\left[\frac 12\frac{r^{20/9}}{(\log r)^{29/9}}\left(1+\frac{\log\log r}{\log r-\log\log r}+O\left(\frac{1}{\log r}\right)\right)\right],\] where $\pi(r)$ is the number of primes less than or equal to $r$. She has also shown that there is no real number $r$ such that $\mathcal N(-r)=4$ (for more about this, see Section 4 below). 

Part (3) of Theorem \ref{Thm2} provides a new natural number-theoretic way of generating countable dense subsets of the plane. Furthermore, these sets come equipped with their own natural enumeration. For example, \[\sigma_{-0.5+i}(1),\sigma_{-0.5+i}(2),\sigma_{-0.5+i}(3),\ldots\] is an enumeration of a dense subset of $\mathbb C$ (possibly with repetitions). 

Part (4) of Theorem \ref{Thm2} is also noteworthy for the following reason. Let us write $c=a+bi$ for $a,b\in\mathbb R$. Imagine fixing $b$ and letting $a$ decrease without bound. The crude estimate \[|\sigma_c(n)-1|=\left|\sum_{1<d\mid n}d^c\right|\leq\sum_{1<d\mid n}|d^c|=\sum_{1<d\mid n}d^a=O(2^a),\] which holds uniformly in $n$, shows that the diameter of the set $\overline{\sigma_c(\mathbb N)}$ decays at least exponentially in $a$ as $a\to-\infty$. Part (4) of the theorem tells us that while $\overline{\sigma_c(\mathbb N)}$ is ``shrinking" in diameter, it continues to ``bud off" into more and more connected components. 

The preceding paragraph does not exclude the possibility that $\mathcal N(c)=\infty$ for some complex numbers $c$. Indeed, parts (1) and (2) of Theorem \ref{Thm2} tell us that $\mathcal N(c)=\infty$ when $\Re(c)>0$ or $c=0$. According to Theorem \ref{Thm1} and part (3) of Theorem \ref{Thm2}, $\mathcal N(c)=1$ if $-1\leq \Re(c)\leq 0$ and $c\neq 0$. This leads us to the central question of this article. 

\begin{question}\label{Quest1}
Suppose $\Re(c)<-1$. Does the set $\overline{\sigma_c(\mathbb N)}$ have finitely many connected components? 
\end{question} 

As mentioned above, Sanna answered this question in the affirmative in the case in which $c$ is real \cite{Sanna}. The primary purpose of this article is to prove the following theorem, fully answering Question \ref{Quest1}.   

\begin{theorem}\label{Thm3}
If $\Re(c)\leq 0$ and $c\neq 0$, then $\overline{\sigma_c(\mathbb N)}$ has finitely many connected components. 
\end{theorem}

Along the way, we prove that the topological subspace $\overline{\sigma_c(\mathbb N)}$ of $\mathbb C$ has nonempty interior (this is equivalent to the claim that $\overline{\sigma_c(\mathbb N)}$ contains a disk of positive radius) when $\Re(c)<-1$ and $c\not\in\mathbb R$ (see Theorem \ref{Thm4} below). This seemingly innocuous claim, whose proof is apparently more elusive than one might expect, is key to proving Theorem \ref{Thm3}.

We already know by Sanna's result and part (3) of Theorem \ref{Thm2} that Theorem \ref{Thm3} is true when $c\in\mathbb R$ or $-1\leq\Re(c)\leq 0$. Indeed, if $-1\leq\Re(c)\leq 0$ and $c\not\in\mathbb R$, then $\overline{\sigma_c(\mathbb N)}=\mathbb C$ is connected! Therefore, we may assume that $c\in\mathbb C\setminus\mathbb R$ and $\Re(c)<-1$ in the next section. Since $\sigma_{\overline{c\raisebox{1.4mm}{}}}(n)=\overline{\sigma_c(n)}$ for all positive integers $n$, we assume without loss of generality that $\Im(c)>0$ (in this one instance, bars denote complex conjugation, not topological closure).   

The paper is organized as follows. In Section 2, we collect some technical yet useful information about the geometry of the spiral $\mathcal S_c=\{\log(1+t^c):t\geq T_c\}$ (for some fixed $T_c\geq 1$) and its close approximation $\widetilde{\mathcal S}_c=\{t^c:t\geq T_c\}$. The reason for doing this comes from the simple fact that $\log(\sigma_c(p))$ lies on $\mathcal S_c$ for all sufficiently large primes $p$. In Section 3, we prove that $\overline{\sigma_c(\mathbb N)}$ contains a disk of positive radius centered at $1$. The proof of Theorem \ref{Thm3} then boils down to a relatively painless argument showing that the connected components of $\overline{\sigma_c(\mathbb N)}$ contain disks whose radii are uniformly bounded away from $0$ (this proves the desired result since part (1) of Theorem \ref{Thm2} tells us that $\overline{\sigma_c(\mathbb N)}$ is bounded). In Section 4, we list some open problems and conjectures concerning the results from this paper and the results from Zubrilina's recent paper. 

\section{Spiral Geometry}

The key property of the divisor function $\sigma_c$ that we exploit is that it is the unique multiplicative arithmetic function that satisfies $\sigma_c(p^\alpha)=1+p^c+p^{2c}+\cdots+p^{\alpha c}$ for all primes $p$ and positive integers $\alpha$. In particular, $\sigma_c(p)=1+p^c$. This suggests that it is useful to understand the geometric properties of the spiral-shaped curve $\{1+t^c\colon t>0\}$. It will often be useful to convert multiplication to addition through logarithms, so we define the principal value logarithm by $\log(z)=\log|z|+\arg(z)\,i$. We use the convention that $\arg(z)$, the principal argument of $z$, lies in the interval $(-\pi,\pi]$. 

Throughout this section, let $c=a+bi$, where $a<-1$ and $b>0$. Since $\Re(1+t^c)>0$ when $t\geq 1$, we may safely define \[\ell_c(t)=\log(1+t^c)\] for all $t\geq 1$. Consider the curve traced out in the complex plane by $\ell_c(t)$ as $t$ increases from $1$ to $\infty$. For each $t\geq 1$, $\ell_c'(t)$ is a vector tangent to this curve at the point $\ell_c(t)$ (we associate each complex number $z$ with the vector $(\Re(z),\Im(z))\in\mathbb R^2$). Let $\lambda(t)=\arg\left(\ell_c'(t)/\ell_c(t)\right)$ be the angle between the vector $\ell_c(t)$ and the tangent vector $\ell_c'(t)$. If $\pi/2<\lambda(t_0)<\pi$, then the vector $\ell_c(t)$ is rotating counterclockwise and decreasing in magnitude when $t=t_0$. We wish to show that there exists $T_c\geq 1$ such that $\pi/2<\lambda(t)<\pi$ for all $t\geq T_c$. This will show that the curve traced out by $\ell_c(t)$ for $t\geq T_c$ is a true ``spiral" in the sense that $\ell_c(t)$ rotates counterclockwise and decreases in magnitude as $t$ increases. Once we know that $T_c$ exists, we let 
\begin{equation}
\mathcal S_c=\{\ell_c(t)\colon t\geq T_c\} 
\end{equation}
be the spiral traced out by $\ell_c(t)$. 

As $t\to\infty$, we have \[\frac{\ell_c'(t)}{\ell_c(t)}=\frac{ct^{c-1}}{(1+t^c)\log(1+t^c)}=\frac{ct^{c-1}}{(1+t^c)(t^c+O(t^{2c}))}=\frac{c}{t+O(t^{c+1})}=\frac{c}{t+o(1)}.\] This shows that 
\[\lambda(t)=\arg(c)-\arg(t+o(1))=\arg(c)+o(1)\] as $t\to\infty$. We have assumed that $\Re(c)<-1$ and $\Im(c)>0$, so $\pi/2<\arg(c)<\pi$. It follows that $\pi/2<\lambda(t)<\pi$ for all sufficiently large $t$, as desired. 

The true purpose of this section is to develop a small theory of the spiral $\mathcal S_c$ for use in the proof of Theorem \ref{Thm3}. 

Let $[z_1,z_2]$ denote the line segment in the complex plane with endpoints $z_1$ and $z_2$. In symbols, $[z_1,z_2]=\{xz_1+(1-x)z_2\colon 0\leq x\leq 1\}$. As usual, we might write $(z_1,z_2]$, $[z_1,z_2)$, or $(z_1,z_2)$ depending on which endpoints we wish to include in the set. For any nonzero complex number $z$, let $w(z)$ be the point in the set $\{\xi\in \mathcal S_c\colon \arg(\xi)=\arg(z),|\xi|<|z|\}$ with the largest absolute value. Geometrically, $w(z)$ is the point of intersection of the spiral $\mathcal S_c$ and the line segment $[0,z)$ that is closest to $z$. Viewing $w\colon\mathbb C\setminus\{0\}\to\mathcal S_c$ as a function, we define the iterates $w^j$ by $w^1(z)=w(z)$ and $w^{j+1}(z)=w(w^j(z))$ for $j\geq 1$. For $\xi\in\mathcal S_c$, let $\theta(\xi)$ be the angle subtended from the origin by the arc of the spiral $\mathcal S_c$ with endpoints $\ell_c(T_c)$ and $\xi$. For example, $\theta(\ell_c(T_c))=0$, and $\theta(w^j(z))=\theta(z)+2j\pi$ for every $z\in\mathcal S_c$ and $j\in\mathbb N$.  

In order to make use of the estimate $\ell_c(t)=t^c+O(t^{2c})$, we let \[\widetilde{\mathcal S}_c=\{t^c\colon t\geq T_c\}.\] For $z\neq 0$, let $\widetilde w(z)$ be the point in the set $\{\xi\in\widetilde{\mathcal S}_c\colon\arg(\xi)=\arg(z),|\xi|<|z|\}$ with the largest absolute value. Let $\widetilde w^1=\widetilde w$ and $\widetilde w^{j+1}=\widetilde w\circ\widetilde w^j$ for $j\geq 1$. Let $\widetilde\theta(\xi)$ be the angle subtended from the origin by the arc of the spiral $\widetilde{\mathcal S}_c$ with endpoints $T_c^c$ and $\xi$. For example, $\widetilde\theta(\widetilde w^j(z))=\widetilde\theta(z)+2j\pi$ for every $z\in\widetilde{\mathcal S}_c$ and $j\in\mathbb N$. It is worth noting that $\widetilde w(t^c)=(e^{2\pi/b}t)^c$ for all $t\geq T_c$. Let $\widetilde A_t=\{x^c\colon t\leq x<e^{2\pi/b}t\}$ be the arc of $\widetilde S_c$ that starts at $t^c$ and makes one full revolution about the origin. For any $z\neq 0$, let $\widetilde g_t(z)$ be the unique point on $A_t$ with the same argument as $z$.  

\begin{figure}[t]
\begin{center}
\includegraphics[height=10cm]{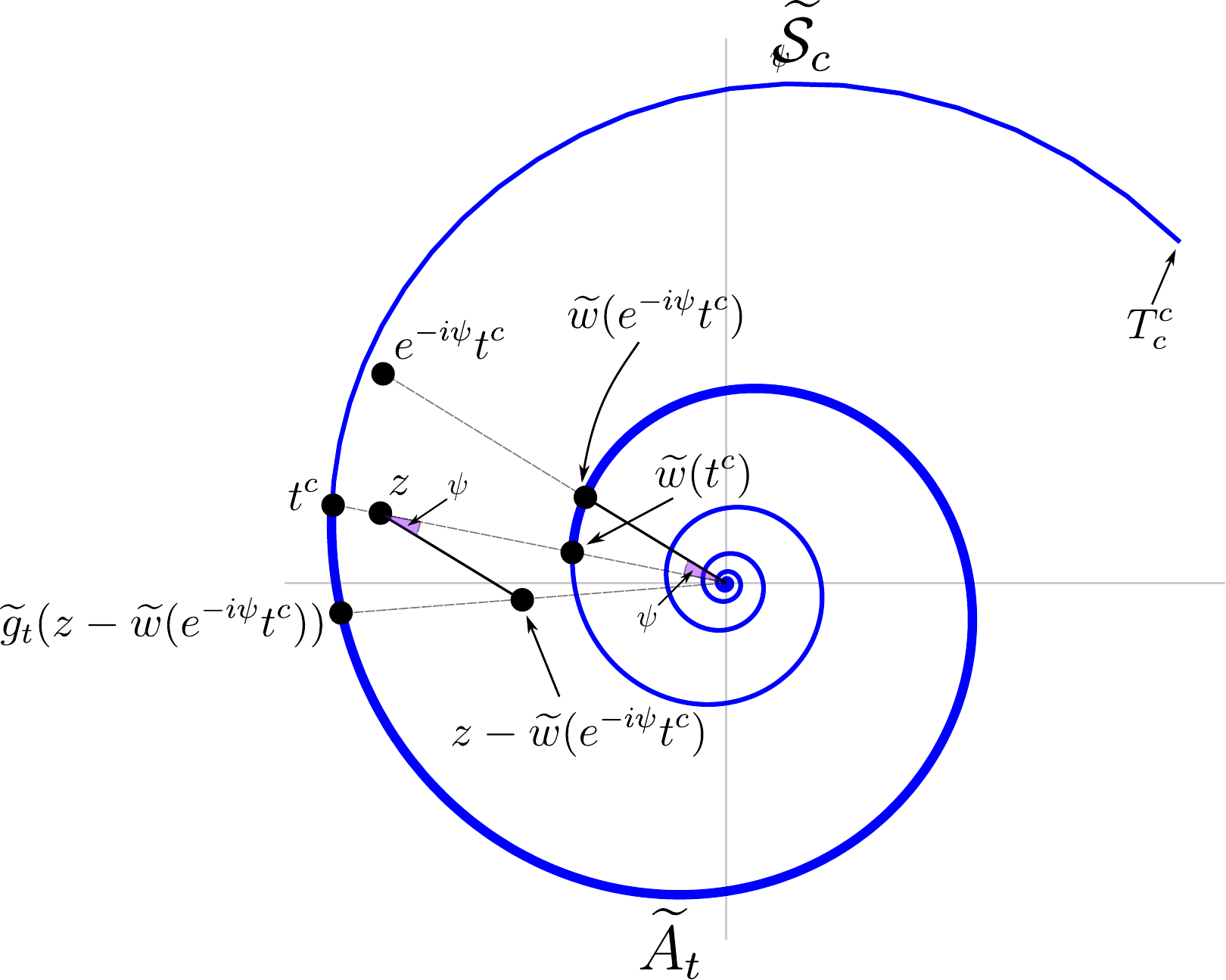}
\end{center}
\captionof{figure}{An illustration of the spiral $\widetilde{\mathcal S}_c$. Here, $\psi$ is an angle measurement satisfying $0\leq\psi\leq\psi_0$. We have chosen $t>T_c$ arbitrarily and selected $z\in(\widetilde w(t^c),t^c]$. The arc $\widetilde A_t$ is drawn in bold. Notice that, as claimed, $|z-\widetilde w(e^{-i\psi}t^c)|\leq|\widetilde g_t(z-\widetilde w(e^{-i\psi}t^c))|$.} \label{Fig1}
\end{figure}

If $0\leq \psi<\pi$ and $t>T_c$, then $\widetilde w(e^{-i\psi}t^c)=(e^{(2\pi-\psi)/b}t)^c$. Elementary geometric arguments show that there exists $\widetilde \psi_0>0$ such that \[|z-\widetilde w^j(e^{-i\psi}t^c)|\leq|\widetilde g_t(z-\widetilde w^j(e^{-i\psi}t^c))|\] whenever $j\in\mathbb N$, $0\leq\psi\leq\widetilde \psi_0$, $t>T_c$, and $z\in(\widetilde w(t^c),t^c]$. See Figure \ref{Fig1} for an illustration of this claim with $j=1$. It follows that if $j\in\mathbb N$, $0\leq\psi\leq\psi_0$, $t>T_c$, and $z\in(\widetilde w(t^c),t^c]$, then \[\widetilde\theta\left(\widetilde w\left(z-\widetilde w^j(e^{-i\psi}t^c)\right)\right)\geq\widetilde\theta(\widetilde w(z)).\] This leads us to the following lemma. 

\begin{lemma}\label{Lem3}
Suppose $c=a+bi$, where $a<-1$ and $b>0$. With notation as above, there exist $\tau_0\geq T_c$ and $\psi_0>0$ such that \[\theta\left(w\left(z-w^j(e^{-i\psi}\ell_c(t))\right)\right)\geq\theta(w(z))\] for all $j,\psi,t,z$ with $j\in\mathbb N$, $0\leq\psi\leq\psi_0$, $t\geq \tau_0$, and $z\in(w(\ell_c(t)),\ell_c(t)]$.  
\end{lemma} 

\begin{proof}
The proof amounts to combining the discussion from the previous paragraph with the estimate $\ell_c(t)=t^c+O(t^{2c})$. The additional parameter $\tau_0$ is introduced in order to ensure that $t$ is large enough for this estimate to be useful. The necessary calculations are straightforward yet mildly annoying, so we omit the details. 
\end{proof}

Before proceeding to the proof of Theorem \ref{Thm3}, we gather one more technical lemma. 

\begin{lemma}\label{Lem4}
Let $c=a+bi$, where $a<-1$ and $b>0$. With notation as above, there exist $\chi_1,\chi_2$ such that $0<\chi_1\leq\chi_2<1$ and $\chi_1\leq|w(\ell_c(t))|/|\ell_c(t)|\leq\chi_2$ for all $t\geq T_c$. 
\end{lemma}
\begin{proof}
The map $f\colon[T_c,\infty)\to\mathbb R$ given by $f(t)=|w(\ell_c(t))|/|\ell_c(t)|$ is continuous. Note that $|w(t^c)|/|t^c|=e^{2\pi a/b}\in(0,1)$ for all $t>0$. It follows from the estimate $\ell_c(t)=t^c+O(t^{2c})$ that $f(t)\to e^{2\pi a/b}$ as $t\to\infty$. Furthermore, $0<f(t)<1$ for all $t\geq T_c$. Setting $\chi_1=\inf_{t\geq T_c}f(t)$ and $\chi_2=\sup_{t\geq T_c}f(t)$, we find that $0<\chi_1\leq\chi_2<1$. 
\end{proof}

\section{The Proof of Theorem \ref{Thm3}}

In what follows, let $\mathbb N$, $\mathbb N_0$, $\mathcal P$, and $SF$ denote the set of positive integers, the set of nonnegative integers, the set of prime numbers, and the set of squarefree positive integers, respectively. Let $\mathbb N_{\leq y}$ and $\mathbb N_{>y}$ denote the set of positive integers whose prime factors are all at most $y$ (these are the so-called $y$-\emph{smooth} numbers) and the set of positive integers whose prime factors are all greater than $y$ (the $y$-\emph{rough} numbers). Put \[\mathcal P_{\leq y}=\mathcal P\cap\mathbb N_{\leq y},\hspace{.2cm}\mathcal P_{>y}=\mathcal P\cap\mathbb N_{>y},\hspace{.2cm}\mathcal SF_{\leq y}=\mathcal SF\cap\mathbb N_{\leq y},\hspace{.2cm}\text{and}\hspace{.2cm}\mathcal SF_{>y}=\mathcal SF\cap\mathbb N_{>y}.\] The notions of $y$-smooth and $y$-rough numbers are important for our discussion because, roughly speaking, the connected component of $\overline{\sigma_c(\mathbb N)}$ containing the point $\sigma_c(n)$ is determined by the small prime factors that divide $n$. Indeed, $\sigma_c(n)$ does not change by much if we multiply $n$ by a power of a very large prime. For $z\in\mathbb C$ and $r>0$, let $B_r(z)=\{\xi\in\mathbb C\colon|\xi-z|<r\}$ be the open disk of radius $r$ centered at $z$. 

We are going to apply Lemma \ref{Lem3} to show that $\overline{\sigma_c(\mathbb N)}$ has nonempty interior. In fact, we can prove the following stronger theorem. 

\begin{theorem}\label{Thm4}
Let $c=a+bi$, where $a<-1$ and $b\neq 0$. If $y$ is sufficiently large, then \[B_{2y^a}(1)\subseteq \overline{\sigma_c(\text{SF}_{>y})}.\]  
\end{theorem}

Before proceeding with the proof of Theorem \ref{Thm4}, let us see how to deduce Theorem \ref{Thm3} from it. We first need one more lemma. For $A,B\subseteq\mathbb C$ and $z\in\mathbb C$, let $AB=\{ab\colon a\in A,b\in B\}$ and $zA=\{za\colon a\in A\}$. 

\begin{lemma}\label{Lem2}
If $\Re(c)<-1$, then there exists $\delta_c>0$ such that $|\sigma_c(n)|\geq\delta_c$ for all $n\in\mathbb N$. 
\end{lemma}
\begin{proof}
Fix a complex number $c=a+bi$, where $a,b\in\mathbb R$ and $a<-1$. For each prime $p$, let $\kappa(p)=\inf\{|\sigma_c(p^\alpha)|\colon\alpha\in\mathbb N_0\}$. Let $\displaystyle\delta_c=\prod_{p\text{ prime}}\kappa(p)$. If $\nu_p(x)$ denotes the exponent of the prime $p$ in the prime factorization of a positive integer $x$, then $\displaystyle|\sigma_c(n)|=\prod_{p\text{ prime}}\left|\sigma_c\left(p^{\nu_p(n)}\right)\right|\geq\delta_c$ for every $n\in\mathbb N$. Therefore, it suffices to show that $\delta_c\neq 0$.   

For any prime $p$ and nonnegative integer $\alpha$, we have \[\sigma_c(p^\alpha)=\frac{p^{(\alpha+1)c}-1}{p^c-1}\neq 0\] because $|p^{(\alpha+1)c}|=p^{(\alpha+1)a}<1$. We also know that $\displaystyle\lim_{\alpha\to\infty}|\sigma_c(p^\alpha)|=\frac{1}{|1-p^c|}\neq 0$. Therefore, $\kappa(p)\neq 0$. We are left to show that the product defining $\delta_c$ does not diverge to $0$, which amounts to showing that the series $\displaystyle\sum_{p\text{ prime}}(\kappa(p)-1)$ converges. Since $a<-1$ and \[|\sigma_c(p^\alpha)|=|1+p^c+p^{2c}+\cdots+p^{\alpha c}|=1+O(p^a),\] we have \[\sum_{p\text{ prime}}(\kappa(p)-1)=O\left(\sum_{p\text{ prime}}p^a\right)<\infty,\] as desired.      
\end{proof}

\begin{proof}[Proof of Theorem \ref{Thm3} assuming Theorem \ref{Thm4}] 
Fix a nonzero complex number $c=a+bi$, where $a,b\in\mathbb R$ and $a\leq 0$. We wish to show that $\overline{\sigma_c(\mathbb N)}$ has finitely many connected components. This follows from part (3) of Theorem \ref{Thm2} if $a\geq -1$, so we may assume $a<-1$. Sanna has proven that $\overline{\sigma_c(\mathbb N)}$ has finitely many connected components if $a<-1$ and $b=0$ \cite{Sanna}. Thus, we may assume $b\neq 0$. 

Let $\delta_c$ be as in Lemma \ref{Lem2}. According to Theorem \ref{Thm4}, there exists $\widehat{y}\geq 1$ such that 
\begin{equation}\label{Eq11}
B_{2y^a}(1)\subseteq\overline{\sigma_c(SF_{>y})}\hspace{0.5cm}\text{for all}\hspace{0.5cm}y\geq \widehat{y}. 
\end{equation}
We may also assume that $\widehat{y}{\hspace{.03cm}}^a<1/3$. Let $n\in\mathbb N_{\leq \widehat{y}}$, and let $C$ be the connected component of $\overline{\sigma_c(\mathbb N)}$ such that $\sigma_c(n)\in C$. Every element of $SF_{>\widehat{y}}$ is relatively prime to $n$, so it follows from the multiplicativity of $\sigma_c$ that $\sigma_c(n)\sigma_c(SF_{>\widehat{y}})\subseteq\sigma_c(\mathbb N)$. Using \eqref{Eq11} and the definition of $\delta_c$, we find that \[\overline{\sigma_c(\mathbb N)}\supseteq\sigma_c(n)\overline{\sigma_c(SF_{>\widehat{y}})}\supseteq\sigma_c(n)B_{2\widehat{y}{\hspace{.02cm}}^a}(1)=B_{2\widehat{y}{\hspace{.02cm}}^a|\sigma_c(n)|}(\sigma_c(n))\supseteq B_{2\widehat{y}{\hspace{.02cm}}^a\delta_c}(\sigma_c(n)).\] This shows that every connected component of $\overline{\sigma_c(\mathbb N)}$ that contains an element of $\sigma_c(\mathbb N_{\leq \widehat{y}})$ must contain a disk of radius $2\widehat{y}{\hspace{.03cm}}^a\delta_c$. Part (1) of Theorem \ref{Thm2} tells us that $\overline{\sigma_c(\mathbb N)}$ is bounded, so there are only finitely many connected components of $\overline{\sigma_c(\mathbb N)}$  that contain elements of $\sigma_c(\mathbb N_{\leq \widehat{y}})$. Let $\mathcal C_1,\mathcal C_2,\ldots,\mathcal C_k$ be these connected components. We will show that $\overline{\sigma_c(\mathbb N)}$ has no other connected components. 

Say a real number $x\geq 1$ is \emph{special} if \[\sigma_c(\mathbb N_{\leq x})\subseteq\bigcup_{j=1}^k\mathcal C_j.\] Our goal is to show that every real number $x\geq 1$ is special. Suppose otherwise. The collection of special numbers is of the form $[1,p)$ for some prime $p$. The preceding paragraph tells us that $\widehat{y}$ is special, so $\widehat{y}<p$. There exists a positive integer $m$ such that $p$ is the largest prime factor of $m$ and $\sigma_c(m)\not\in\bigcup_{j=1}^k\mathcal C_j$. Write $m=m'p^\alpha$, where $\alpha$ and $m'$ are positive integers and $p\nmid m'$. Because $m'\in\mathbb N_{\leq p-1}$ and $p-1$ is special, there exists $r\in\{1,2,\ldots,k\}$ such that $\sigma_c(m')\in\mathcal C_r$. We will derive a contradiction by showing that $\sigma_c(m)\in\mathcal C_r$. 

Let $h=\max\{\widehat{y},p-1\}$. Every element of $SF_{>h}$ is relatively prime to $m'$, so it follows from \eqref{Eq11} and the multiplicativity of $\sigma_c$ that \[\overline{\sigma_c(\mathbb N)}\supseteq\sigma_c(m')\overline{\sigma_c(SF_{>h})}\supseteq\sigma_c(m')B_{2h^a}(1)=B_{2h^a|\sigma_c(m')|}(\sigma_c(m')).\] Since $B_{2h^a|\sigma_c(m')|}(\sigma_c(m'))$ is a connected subset of $\overline{\sigma_c(\mathbb N)}$ that contains $\sigma_c(m')$, we know that $B_{2h^a|\sigma_c(m')|}(\sigma_c(m'))\subseteq\mathcal C_r$. Thus, it suffices to show that $\sigma_c(m)\in B_{2h^a|\sigma_c(m')|}(\sigma_c(m'))$. 

Because $\widehat{y}<p$ and $\widehat{y}{\hspace{.03cm}}^a<1/3$, we have \[\left|\frac{p^{\alpha c}-1}{p^c-1}\right|\leq\frac{1+p^{\alpha a}}{1-p^a}\leq\frac{1+p^a}{1-p^a}<\frac{1+\widehat{y}{\hspace{.03cm}}^a}{1-\widehat{y}{\hspace{.03cm}}^a}<\frac{1+1/3}{1-1/3}=2.\] Therefore, \[\frac{|\sigma_c(m)-\sigma_c(m')|}{|\sigma_c(m')|}=|\sigma_c(p^\alpha)-1|=\left|p^c+p^{2c}+\cdots+p^{\alpha c}\right|=|p^c|\left|\frac{p^{\alpha c}-1}{p^c-1}\right|<2p^a.\] We know that $p^a<h^a$ because $p>h$. Consequently, \[|\sigma_c(m)-\sigma_c(m')|<2p^a|\sigma_c(m')|<2h^a|\sigma_c(m')|.\] In other words, $\sigma_c(m)\in B_{2h^a|\sigma_c(m')|}(\sigma_c(m'))$.    
\end{proof}

\begin{proof}[Proof of Theorem \ref{Thm4}]
By taking complex conjugates, we see that $\sigma_{a-bi}(SF_{>y})$ is obtained by reflecting $\sigma_{a+bi}(SF_{>y})$ through the real axis. This observation allows us to assume $b>0$ without loss of generality. Preserve the notation from Section 2. In particular, let $\tau_0$ and $\psi_0$ be as in Lemma \ref{Lem3}. We may assume $\psi_0<\pi/2$. For all sufficiently large $y$, we have $B_{2y^a}(1)\subseteq\exp(B_{3y^a}(0))$. Therefore, it suffices to show that $B_{3y^a}(0)\subseteq\overline{\log(\sigma_c(SF_{>y}))}$ for all sufficiently large $y$. 

The main ideas of the proof are as follows. After deciding how large we must choose $y$, we choose an appropriate positive integer $j$ and select $z_0\in B_{3y^a}(0)$. We then form a sequence $(z_k)_{k\geq 0}$ in the following manner. Assume we have defined $z_k$, and put $z_{k+1}=z_k-\sum_{n=1}^{N_k}\ell_c\left(q_k^{(n)}\right)$ for some distinct primes $q_k^{(n)}$ satisfying $0\leq \theta(w^j(z_k))-\theta\left(\ell_c\left(q_k^{(n)}\right)\right)<\psi_0/2$. With the help of Lemma \ref{Lem3}, we show that $\theta(w(z_k))\to\infty$ as $k\to\infty$. This then implies that $z_k\to 0$. We will show that it is possible to choose the primes of the form $q_k^{(n)}$ so that they are all distinct and at least $y$. It will then follow that $z_0-z_k\in\log(\sigma_c(SF_{>y}))$ for all $k$, completing the proof. We now present the details of this argument. The reader might find it helpful to refer to Figure \ref{Fig3} while trudging onward through the proof.      

Let $u=e^{\psi_0/(3b)}$. Using the fact that $\ell_c(t)=t^c+O(t^{2c})$, one can verify that $\arg(\ell_c(u\tau)/\ell_c(\tau))\to\psi_0/3$ as $\tau\to\infty$. Consequently, there exists $\tau_0'\geq 0$ such that 
\begin{equation}\label{Eq9}
\arg(\ell_c(u\tau)/\ell_c(\tau))<\psi_0/2\hspace{0.5cm}\text{for all}\hspace{0.5cm}\tau\geq \tau_0'.
\end{equation} 

Let $\chi_1$ and $\chi_2$ be as in Lemma \ref{Lem4}, and fix a positive integer $j$ such that 
\begin{equation}\label{Eq2}
\chi_2^{j-1}\leq u^a/4.
\end{equation} Fix an integer 
\begin{equation}\label{Eq3}
M>\chi_1^{-j}\sec(\psi_0/2).
\end{equation} It is well-known that $\frac{p_{m+1}}{p_m}\to 1$ as $m\to\infty$, where $p_m$ denotes the $m^{\text{th}}$ prime number. Therefore, there exists $\tau_0''>0$ such that $|\mathcal P\cap[\tau,u\tau]|\geq M$ for all $\tau\geq \tau_0''$. Since $\ell_c(t)=t^c+O(t^{2c})$, there exists $\tau_0'''$ such that 
\begin{equation}\label{Eq8}
\frac34|\tau^c|<|\ell_c(\tau)|\hspace{0.5cm}\text{for all}\hspace{0.5cm}\tau\geq \tau_0'''.
\end{equation} Choose $y$ large enough so that 
\begin{equation}\label{Eq1}
\xi\in B_{3y^a}(0)\Longrightarrow w(\xi)=w^3(\ell_c(x))\hspace{0.5cm}\text{for some}\hspace{0.5cm}x\geq\max\{\tau_0,\tau_0',\tau_0'',\tau_0'''\}.
\end{equation} We may also assume $y$ is large enough so that $\ell_c(y)\in B_{3y^a}(0)$. This forces $y\geq\max\{\tau_0,\tau_0',\tau_0'',\tau_0'''\}$. 

\begin{figure}[t]
\begin{center}
\includegraphics[height=12cm]{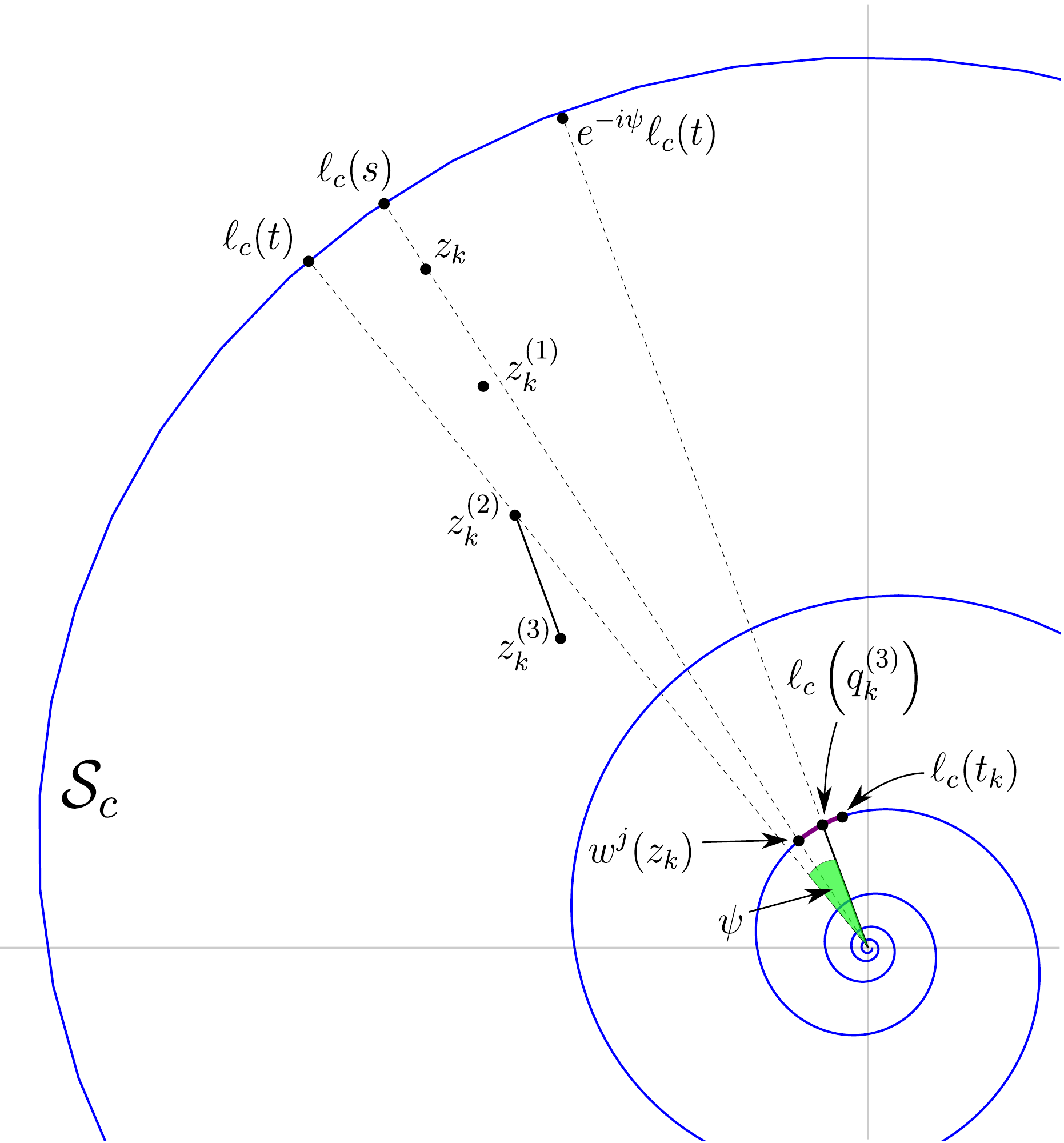}
\end{center}
\captionof{figure}{An illustration of part of the construction of $z_{k+1}$ from $z_k$. We have drawn $z_k^{(m)}$ for $m\in\{0,1,2,3\}$, although it is certainly possible that $M>3$. In this figure, we have taken $j=2$. We have also assumed $n=2$ in \eqref{Eq10}, which is why $\arg(\ell_c(t))=\arg\left(z_k^{(2)}\right)$. Notice that $z_k^{(3)}=z_k^{(2)}-\ell_c\left(q_k^{(3)}\right)$. The arc $\ell_c([t_k,ut_k])$ is colored purple. Observe that $\ell_c\left(q_k^{(3)}\right)$ lies on this purple arc since $q_k^{(3)}\in[t_k,ut_k]$. } \label{Fig3}
\end{figure}

Pick $z_0\in B_{3y^a}(0)$. Assume we have already defined $z_k\in B_{3y^a}(0)$. We define $z_{k+1}$ as follows. Let $t_k$ be the unique number such that $t_k\geq T_c$ and $\ell_c(ut_k)=w^j(z_k)$. According to \eqref{Eq2} and the definition of $\chi_2$, \[|w^j(z_k)|=|w^{j-1}(w(z_k))|\leq\chi_2^{j-1}|w(z_k)|<\chi_2^{j-1}|z_k|\leq u^a|z_k|/4.\] We have assumed that $z_k\in B_{3y^a}(0)$, so $|w^j(z_k)|<\frac 34u^ay^a$. Since $uy>y\geq \tau_0'''$, \eqref{Eq8} tells us that $\frac 34u^ay^a=\frac 34|(uy)^c|<|\ell_c(uy)|$. Hence, 
\[|\ell_c(ut_k)|=|w^j(z_k)|<|\ell_c(uy)|.\] This implies that $ut_k>uy$, so $t_k>y$. Since $t_k>y\geq \tau_0''$, it follows from the choice of $\tau_0''$ that there are distinct primes $q_k^{(1)},q_k^{(2)},\ldots,q_k^{(M)}\in[t_k,ut_k]$. For each $m\in\{0,1,\ldots,M\}$, let 
\begin{equation}\label{Eq7}
z_k^{(m)}=z_k-\sum_{n=1}^m\ell_c\left(q_k^{(n)}\right).
\end{equation} In particular, $z_k^{(0)}=z_k$. 

To ease notation, let $\phi_k^{(n)}=\theta\left(w\left(z_k^{(n)}\right)\right)$. Suppose 
\begin{equation}\label{Eq10}
0\leq \phi_k^{(n)}-\phi_k^{(0)}\leq\psi_0/2
\end{equation} for some $n\in\{0,1,\ldots,M-1\}$. Let $t$ be the unique real number such that $t\geq T_c$ and $z_k^{(n)}\in(w(\ell_c(t)),\ell_c(t)]$. Because $z_k\in B_{3y^a}(0)$, it follows from \eqref{Eq1} and \eqref{Eq10} that $t\geq \tau_0$. Let \[\psi=\arg\left(\ell_c(t)\Big/\ell_c\left(q_k^{(n)}\right)\right),\hspace{0.5cm}\psi'=\arg\left(w^j(z_k)\Big/\ell_c\left(q_k^{(n)}\right)\right),\hspace{0.5cm}\text{and}\] \[\psi''=\arg\left(\ell_c(t)/w^j(z_k)\right).\]
We know that $w^j(z_k)=\ell_c(ut_k)$, $t_k>y\geq \tau_0'$, and $q_k^{(n)}\in[t_k,ut_k]$. Therefore, \eqref{Eq9} tells us that $0\leq\psi'<\psi_0/2$. Since $\arg(w^j(z_k))=\arg(w(z_k))$ and $\arg(\ell_c(t))=\arg\left(w\left(z_k^{(n)}\right)\right)$, we have $\psi''=\phi_k^{(n)}-\phi_l^{(0)}$. According to \eqref{Eq10}, $0\leq\psi''\leq\psi_0/2$. Consequently, $0\leq\psi=\psi'+\psi''<\psi_0$. Observe that $\ell_c\left(q_k^{(n)}\right)=w^j(e^{-i\psi}\ell_c(t))$. Setting $z=z_k^{(n)}$ in Lemma \ref{Lem3}, we find that \[\phi_k^{(n+1)}=\theta\left(w\left(z_k^{(n)}-w^j(e^{-i\psi}\ell_c(t))\right)\right)\geq\phi_k^{(n)}.\]
This shows that one of the following must hold: 
\begin{enumerate}[(A)]
\item There exists $N_k\in\{1,2,\ldots,M\}$ such that \[\phi_k^{(0)}\leq\phi_k^{(1)}\leq\cdots\leq\phi_k^{(N_k-1)}\leq\phi_k^{(0)}+\psi_0/2<\phi_k^{(N_k)};\]
\item We have $\phi_k^{(0)}\leq\phi_k^{(M)}\leq\phi_k^{(0)}+\psi_0/2$. 
\end{enumerate}

We show that (B) cannot hold. If (B) is true, then $0\leq\arg\left(z_k^{(M)}/z_k\right)\leq\psi_0/2$. Thus, it suffices to show that 
\begin{equation}\label{Eq4}
\pi/2<\arg\left(z_k^{(M)}/z_k\right)\leq\pi\hspace{0.5cm}\text{or}\hspace{0.5cm}-\pi<\arg\left(z_k^{(M)}/z_k\right)\leq-\pi/2. 
\end{equation}
Let $\beta_1=z_k/|z_k|$ and $\beta_2=i\beta_1$. The set $\{\beta_1,\beta_2\}$ is a basis for $\mathbb C$ over $\mathbb R$. Therefore, for every $\xi\in\mathbb C$, there exist unique $\alpha_1(\xi),\alpha_2(\xi)\in\mathbb R$ such that $\xi=\alpha_1(\xi)\beta_1+\alpha_2(\xi)\beta_2$. In order to prove \eqref{Eq4}, we simply need to show that $\alpha_1\left(z_k^{(M)}\right)<0$. 

If $s$ is the unique real number such that $s\geq T_c$ and $w(\ell_c(s))=w(z_k)$ (in other words, $s$ is the unique real number such that $s\geq T_c$ and $z_k\in(w(\ell_c(s)),\ell_c(s)]$), then 
\begin{equation}\label{Eq5}\left|\ell_c\left(q_k^{(n)}\right)\right|\geq|w^j(z_k)|=|w^j(\ell_c(s))|\geq\chi_1^j|\ell_c(s)|\geq\chi_1^j|z_k|
\end{equation} for all $n\in\{1,2,\ldots,M\}$. We saw above (using \eqref{Eq9} and the fact that $q_k^{(n)}\in[t_k,ut_k]$) that \[0\leq \arg\left(w^j(z_k)\Big/\ell_c\left(q_k^{(n)}\right)\right)<\psi_0/2\] for all such $n$. Since $\arg(w^j(z_k))=\arg(z_k)$, this implies that
\begin{equation}\label{Eq6}\alpha_1\left(\ell_c\left(q_k^{(n)}\right)\right)\geq\left|\ell_c\left(q_k^{(n)}\right)\right|\cos(\psi_0/2)
\end{equation} 
for all $n\in\{1,2,\ldots,M\}$. Combining \eqref{Eq3}, \eqref{Eq5}, and \eqref{Eq6} yields \[\sum_{n=1}^M\alpha_1\left(\ell_c\left(q_k^{(n)}\right)\right)\geq\sum_{n=1}^M\chi_1^j|z_k|\cos(\psi_0/2)=\frac{M}{\chi_1^{-j}\sec(\psi_0/2)}|z_k|>|z_k|.\] Using \eqref{Eq7}, we find that \[\alpha_1\left(z_k^{(M)}\right)=\alpha_1(z_k)-\sum_{n=1}^M\alpha_1\left(\ell_c\left(q_k^{(n)}\right)\right)=|z_k|-\sum_{n=1}^M\alpha_1\left(\ell_c\left(q_k^{(n)}\right)\right)<0,\] as desired. 

We have demonstrated that (B) is false, so (A) must be true. Let $z_{k+1}=z_k^{(N_k)}$. It is straightforward to show that \[\left|z_k^{(0)}\right|\geq\left|z_k^{(1)}\right|\geq\cdots\geq\left|z_k^{(N_k)}\right|.\] In particular, $|z_{k+1}|\leq|z_k|<3y^a$. This recursively defines the sequence $(z_k)_{k\geq 0}$ of elements of $B_{3y^a}(0)$. According to (A), $\theta(w(z_{k+1}))>\theta(w(z_k))+\psi_0/2$. We have seen that \[\theta(w^j(z_k))-\psi_0/2<\theta\left(\ell_c\left(q_k^{(n)}\right)\right)\leq\theta(w^j(z_k))\] for all $k\geq 0$ and $n\in\{1,2,\ldots,M\}$. It follows that 
\[\theta\left(\ell_c\left(q_k^{(n)}\right)\right)\leq\theta(w^j(z_k))<\theta(w^j(z_{k+1}))-\psi_0/2<\theta\left(\ell_c\left(q_{k+1}^{(n')}\right)\right)\] for every $k\geq 0$ and all $n,n'\in\{1,2,\ldots,M\}$. Consequently, the primes $q_k^{(n)}$ for $k\geq 0$ and $1\leq n\leq M$ are all distinct. We also saw that $q_k^{(n)}\geq t_k>y$ for all $k$ and all $n$. This shows that \[z_0-z_m=\sum_{k=0}^{m-1}\sum_{n=1}^{N_k}\ell_c\left(q_k^{(n)}\right)=\log\left(\sigma_c\left(\prod_{k=0}^{m-1}\prod_{n=1}^{N_k}q_k^{(n)}\right)\right)\in\log(\sigma_c(SF_{>y}))\] for all $m\geq 1$. Because $\theta(w(z_{k+1}))>\theta(w(z_k))+\psi_0/2$ for all $k\geq 0$, we know that $\theta(w(z_m))\to\infty$ as $m\to\infty$. It follows that $\displaystyle\lim_{m\to\infty}z_m=0$. This completes the proof that $z_0\in\overline{\log(\sigma_c(SF_{>y}))}$. 
\end{proof}

We have made no attempt to optimize Theorem \ref{Thm4}. For instance, for any fixed $K>0$, it is possible to adjust the constants throughout the proof of Theorem \ref{Thm4} in order to show that $B_{Ky^a}(1)\subseteq\overline{\sigma_c(SF_{>y})}$ for all sufficiently large $y$. We have also made no attempt to determine exactly how large ``sufficiently large" is. This leads to the interesting problem, originally suggested by Andrew Kwon, of determining the largest $R$ such that $B_R(1)\subseteq\overline{\sigma_c(\mathbb N)}$. In any event, our original formulation of Theorem \ref{Thm4} is strong enough for the proof Theorem \ref{Thm3}. 

\section{Conclusion and Open Problems}

Suppose $\Re(c)\leq 0$ and $c\neq 0$. We have shown that $\overline{\sigma_c(\mathbb N)}$ has finitely many connected components, and we have shown that this set contains a disk of positive radius if $c\not\in\mathbb R$. Of course, this is a far leap from a thorough understanding of the topological and geometric properties of $\overline{\sigma_c(\mathbb N)}$. For example, we mentioned at the end of the previous section that Andrew Kwon has asked the following interesting question. 

\begin{question}\label{Quest2}
Suppose $c\in\mathbb C\setminus\mathbb R$ and $\Re(c)<-1$. What is the largest real number $R$ such that $B_R(1)\subseteq\overline{\sigma_c(\mathbb N)}$?
\end{question}

Now that we know $\overline{\sigma_c(\mathbb N)}$ has finitely many connected components, it is natural to ask just how many connected components it has! More precisely, we make the following definition. 

\begin{definition}
For each positive integer $m$, let $D_m$ denote the set of all complex numbers $c$ such that $\overline{\sigma_c(\mathbb N)}$ has exactly $m$ connected components. Let $E_m=D_m\cap\mathbb R$. 
\end{definition}

Theorem \ref{Thm1} tells us that $E_1=[-\eta,0)$. We also know from parts (2) and (5) of Theorem \ref{Thm2} that $-3.02<\Re(c)\leq 0 $ if $c\in D_1$ (the lower bound of $-3.02$ is most likely not optimal and could probably be improved without extreme difficulty). Part (3) of the same theorem tells us that $c\in D_1$ if $-1\leq\Re(c)\leq 0$ and $c\neq 0$. However, the following question is still open in general. 

\begin{question}\label{Quest3}
What is $D_1$? 
\end{question}

For a fixed $c$, it is easy for a computer to calculate $\sigma_c(n)$ for several values of $n$. It would be interesting to use techniques and computer programs from the rapidly-developing area of topological data analysis to gain empirical information about the set $D_1$. In particular, one could probably predict the general ``shape" of $D_1$.   

As mentioned in the introduction, Zubrilina has proven asymptotic estimates for $\mathcal N(-r)$, the number of connected components of $\overline{\sigma_{-r}(\mathbb N)}$, when $r>1$ is real \cite{Nina}. These estimates can be translated into information about the sets $E_m$. One could similarly attempt to prove asymptotic estimates for $\mathcal N(c)$ when $c$ is not assumed to be real and $\Re(c)\to-\infty$ (thus obtaining information about the sets $D_m$). In doing so, it might be helpful to assume that $c$ is, in some sense, ``close" to the real axis. Of course, one could also study how $\mathcal N(c)$ changes when $\Re(c)$ is held fixed and $\Im(c)$ varies.      

Zubrilina has proven the somewhat surprising result that $E_4=\emptyset$. For this reason, we make the following definition. 

\begin{definition}
We say a positive integer $m$ is a \emph{Zubrilina number} if $E_m=\emptyset$. We say $m$ is a \emph{strong Zubrilina number} if $D_m=\emptyset$.
\end{definition}

We would like to know if it is possible to give any meaningful description of Zubrilina numbers. In particular, we make the following conjecture.

\begin{conjecture}\label{Conj1}
There are infinitely many Zubrilina numbers. 
\end{conjecture}

This article focuses on general complex divisor functions, so we make the following conjecture involving strong Zubrilina numbers. 

\begin{conjecture}\label{Conj2}
There are infinitely many strong Zubrilina numbers. 
\end{conjecture}
   
Note that Conjecture \ref{Conj2} implies Conjecture \ref{Conj1} because every strong \\ Zubrilina number is a Zubrilina number.  

\section{Acknowledgments}
Recall that Sanna answered Question \ref{Quest1} affirmatively under the assumption that $c$ is real. For two years, the current author naively assumed that the answer to that question is ``yes" in general and that a proof of this fact would require little more than a mild modification of Sanna's argument. It was only after speaking with Nina Zubrilina at the 2017 Duluth Mathematics REU Program that he realized that answering Question \ref{Quest1} is more difficult than he originally thought. Further discussions with participants in the REU program persuaded him to investigate this question more seriously. 

I would like to thank Nina Zubrilina, whose skepticism led me to question my most sacred beliefs about complex divisor functions. I thank Evan Chen, Benjamin Gunby, Andrew Kwon, and Mehtaab Sawhney for helpful discussions and for taking interest in the problem of showing that $\overline{\sigma_c(\mathbb N)}$ has nonempty interior. Finally, I thank Steven J. Miller and the anonymous referee for making helpful comments toward the presentation of this paper.

\end{document}